\documentclass[12pt]{article}
\usepackage[a4paper,left=2cm,right=2cm,top=2cm,bottom=2cm]{geometry}

\usepackage{graphicx}
\graphicspath{{figures/}}

\usepackage[backend=bibtex8,%
]{biblatex}
\usepackage{amsmath,amssymb}
\usepackage{amsthm}
\newtheorem{theorem}{Theorem}
\newtheorem*{theorem*}{Theorem}
\newtheorem{proposition}{Proposition}
\newtheorem{corollary}{Corollary}

\theoremstyle{definition}
\newtheorem{definition}{Definition}

\newtheorem{remark}{Remark}
\newtheorem{example}{Example}
\newtheorem{nonexample}{Non-example}

\usepackage{mathtools}
\usepackage{enumitem}

\usepackage[x11names]{xcolor}

\usepackage{tikz}
\usetikzlibrary{calc}
\usetikzlibrary{
positioning
}

\newcommand{\N}{\mathbb{N}}
\newcommand{\Z}{\mathbb{Z}}

\usepackage{doi}
\usepackage{hyperref}

\title{Additive systems for $\Z$ are undecidable}
\author{Andrei Zabolotskii}
\date{}

\AtEndDocument{%
	\par
	\medskip
	\begin{tabular}{@{}l@{}}%
		{School of Mathematics and Statistics, The Open University,}\\
		{Milton Keynes, MK7 6AA, United Kingdom}\\
		\texttt{andrei.zabolotskii\raisebox{-0.4ex}{{\fontfamily{pag}\selectfont @}}open.ac.uk}
\end{tabular}}

\begin{document}

\maketitle

\begin{abstract}
What are the collections of sets ${A}_i\subset\Z$ such that any $n\in\Z$ has exactly one representation as $n=a_0+a_1+\dotsb$ with $a_i\in{A}_i$? The answer for $\N_0$ instead of $\Z$ is given by a theorem of de Bruijn. We describe a family of natural candidate collections for $\Z$, which we call canonical collections. Translating the problem into the language of dynamical systems, we show that the question of whether the sumset of a canonical collection covers the entire $\Z$ is difficult: specifically, there is a collection for which this question is equivalent to the Collatz conjecture, and there is a well-behaved family of collections for which this question is equivalent to the universal halting problem for Fractran and is therefore undecidable.
\end{abstract}

\section{Introduction}
Let $\mathbb{M}$ be a commutative monoid, with additive notation. Throughout this paper, $\mathbb{M}$ is almost always going to be either the nonnegative integers $\N_0$ or the integers $\Z$.

Let $\mathcal{A}=\{A_0,{A}_1,\dotsc\}$ be a collection of subsets of $\mathbb{M}$, finite or countably infinite. The sets ${A}_i$ themselves also can be finite or infinite. We require that each ${A}_i$ contains 0 and at least one more element.

\begin{definition}\label{def:basic}~\par
\begin{itemize}
\item The \emph{sumset} of $\mathcal{A}$ is the set of sums $a_0+a_1+\dotsb$ where $a_i\in A_i$, for all $i$, and all but finitely many $a_i$ are equal to~0. The sumset of $\mathcal{A}$ is denoted $\sum\mathcal{A}$ or $A_0+A_1+\dotsb$
\item The sumset of $\mathcal{A}$ is \emph{direct} when for any $n\in\sum\mathcal{A}$, its representation as $n=a_0+a_1+\dotsb$ is unique. The direct sumset of $\mathcal{A}$ is denoted $\bigoplus\mathcal{A}$ or $A_0\oplus A_1\oplus\dotsb$
\item An \emph{additive system} for $\mathbb{M}$ is a collection $\mathcal{A}$ such that $\bigoplus\mathcal{A} = \mathbb{M}$.
\end{itemize}
\end{definition}

Additive systems for $\mathbb{M}=\N_0$ are fully described by a foundational result in this area, de Bruijn's Theorem. We state it in Section~\ref{sec:debr}.

In this work, we focus on additive systems for $\mathbb{M}=\Z$, which are notoriously hard to deal with. We define \emph{canonical collections}, which simultaneously generalise in a natural way additive systems for $\N_0$ and basic examples of additive systems for $\Z$. We then switch from the additive combinatorics point of view on canonical collections to the dynamical systems perspective. This makes it possible to construct remarkable and unexpected equivalences between questions from additive combinatorics (whether certain canonical collections are additive systems for $\Z$) and famous difficult and even undecidable problems, namely the Collatz conjecture and the universal halting problem. For the latter problem, we make use of Conway's esoteric programming language Fractran.

We introduce canonical collections in Section~\ref{sec:canonical} and prove our main results, the Collatz conjecture and halting problem equivalences, in~Section~\ref{sec:undecidable} as Theorem~\ref{th:Collatz}, Theorem~\ref{thm:Fractran-type}, and Corollary~\ref{cor:undecidable}. The diagram in Fig.~\ref{fig:Venn} shows the relationship between the sets of collections considered in this paper.

\section{Additive systems for $\N_0$ and de Bruijn's theorem}
\label{sec:debr}

We will denote the integer interval $\{0,1,\dotsc,n-1\}$ for some $n\in\N$ by $[n]$. For a set $A\subset\mathbb{M}$ and $k\in\mathbb{M}$, we will denote the set $\{ka \mid a\in A\}$ by $kA$. We will write $k\mathcal{A}$ for $\{kA_0, kA_1,\dotsc\}$.

It is easy to find an example of an additive system for $\N_0$.
\begin{example}[decimal numeral system]
\label{ex:decimal}~

The collection $\{[10],10[10],100[10],\dotsc,10^k[10],\dotsc\} =$
\[
\begin{aligned}
& \{ \\
&\quad \{0,1,2,3,4,5,6,7,8,9\}, \\
&\quad \{0,10,20,30,40,50,60,70,80,90\}, \\
&\quad \{0,100,200,300,400,500,600,700,800,900\}, \\[-0.2em]
&\quad ... \\[-0.5em]
& \}
\end{aligned}
\]
is an additive system for $\N_0$. Finding the representation of a number as a sum of elements of the sets from this collection is equivalent to computing its decimal expansion; e.g., $538 = \dotsb+0+0+500+30+8$.
\end{example}

\begin{definition}
\label{def:contraction}
Let $\mathcal{A}$ be a collection with a direct sumset and let $\mathcal{B}\subset\mathcal{A}$. The \emph{contraction} of $\mathcal{A}$ with respect to $\mathcal{B}$ is the collection
 $\left(\mathcal{A}\setminus\mathcal{B}\right)\cup\left(\bigoplus\mathcal{B}\right)$.
\end{definition}
The sumset of a contraction is direct too. In particular, when $\mathcal{A}$ is an additive system, then so are all of its contractions.

Example~\ref{ex:decimal} can be generalised to a collection of subsets of $\N_0$ corresponding to a numeral system with a varying base. De Bruijn called such collections ``British number systems'', referring to the systems of weight measurement units like the one presented below as Example~\ref{ex:british}.
\begin{definition}
\label{def:british}
Let $(b_i)_{i\in\N_0}$ be a sequence of integers, called \emph{bases}, such that $b_i>1$ for all $i$.
A \emph{British number system} is the collection $\{A_i\}_{i\in\N_0}$ with $A_i=\left(\prod_{j=0}^{i-1}b_j\right) [b_i]$.

A \emph{numeral system} is a British number system such that its sequence of bases is constant.
\end{definition}

In each of the following examples, only the first few bases are given; after them, the sequence of bases continues in an arbitrary way (e.g.\ becomes constant~10).
\begin{example}~
\label{ex:british}
\begin{figure}[h!]
\centering
\includegraphics[width=0.35\linewidth]{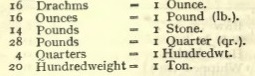}
\caption{\textit{Mrs.\ Beeton's Book of Household Management} (1907), p.~126. $(b_0,b_1,b_2,b_3,b_4,\dotsc) = (16, 16, 28, 4, 20,\dotsc)$}
\end{figure}
\end{example}
\pagebreak
\begin{example}
\label{ex:vis:British}
In a British number system $\mathcal{A}=\{A_0,A_1,\dotsc\}$, let the sequence of bases begin $b_0 = 2$, $b_1=4$, $b_2=3$.

The sets $[b_0],~~[b_1],~~[b_2]$ can be visualised as:
\begin{figure}[h!]
\centering
\includegraphics[width=\linewidth]{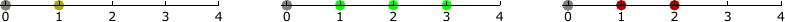}
\end{figure}

The sets $A_0=[b_0],~~A_1=b_0[b_1],~~A_2=b_0b_1[b_2]$, i.e.\ the elements of $\mathcal{A}$ themselves, are:
\begin{figure}[h!]
\centering
\includegraphics[width=\linewidth]{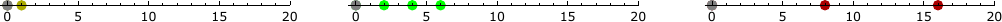}
\end{figure}

Finally, the cumulative sumsets of $\mathcal{A}$ (viewed as an ordered sequence) are \\ the sets $A_0=[b_0],~~A_0+A_1 = [b_0b_1],~~A_0+A_1+A_2 = [b_0b_1b_2]$. They cover an ever-increasing portion of $\N_0$.
\begin{figure}[h!]
\centering
\includegraphics[width=\linewidth]{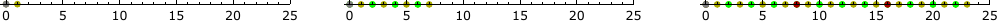}
\end{figure}

The colour of each point $n=a_0+a_1+\dotsb$,~\ $a_i\in A_i$, corresponds to the smallest $i$ such that $a_i\ne0$.
\end{example}

Any British number system or a contraction of such a system is an additive system for $\N_0$. In fact, de Bruijn showed that the converse is also true, thus completely characterising additive systems for $\N_0$.
\begin{theorem*}[de Bruijn \cite{deBruijn}]
\label{thm:deBruijn}
Any additive system for $\N_0$ is a contraction of a British number system.
\end{theorem*}
To prove this, first a lemma is proved: any additive system $\mathcal{A}$ must contain a set $A_0$ of the form $[b]\oplus A_0'$, where $A_0$ is the set that contains~1, which must exist, and $b$ is the smallest number not in $A_0$; and then $\left(\mathcal{A}\setminus\{A_0\}\right)\cup\{A_0'\}$ is of the form $b\mathcal{A}'$ with $\mathcal{A}'$ an additive system. The theorem follows from applying the lemma repeatedly. The previous sentence hides significant technical complexity, as explained in the work \cite{Nathanson} that gives an exposition of the theorem and its proof.

\section{Additive systems for $\Z$ and canonical collections}
\label{sec:canonical}
In this paper, we are primarily concerned with additive systems for $\Z$.
As Nathanson puts it \cite{Nathanson}, 
\begin{quote}
de Bruijn's remark at the end of his 1956 paper on $\N_0$ still accurately describes the current state of the problem: ``$\langle\ldots\rangle$ the~analog\-ous problem for number systems representing uniquely all integers (without restriction to nonnegative ones) $\langle\ldots\rangle$ is~much more difficult than the one dealt with above, and it is still far from a complete solution.''
\end{quote}

The existing literature on additive systems for $\Z$ mostly focuses on the case of additive systems of cardinality~2, i.e.\ $\Z=A\oplus B$. This problem, and the more general problem known as factorisation of abelian groups, have been studied since the 1940s (by de Bruijn and Haj\'os) and to this day. The known results are surveyed in \cite{book:SzaboSands}. Many of them are of the form ``if $\Z=A\oplus B$ and $A,B$ satisfy certain conditions, then there exists another set $A'$ such that $\Z=A'\oplus B$'', e.g.: if $\Z=A\oplus B$ with $A$ finite and $k$ is an integer coprime to $|A|$ then $\Z=(kA)\oplus B$. There exist families of additive systems $\{A,B\}$ for $\Z$, where $A$ is a finite set and $B$ is infinite periodic, with the period of $B$ bounded from above or from below in the terms of the diameter of $A$ \cite{book:SzaboSands}. It is known that for any sparse enough infinite set $A$ of nonnegative integers there exists $B$ such that $\Z=A\oplus B$, so for infinite $A$ and $B$, ``no structure should be expected'' \cite{book:ww}. Finally, the collections that have $[N]$ as the direct sum have been completely described and enumerated \cite{cardiff23}, being essentially finite British number systems and their contractions.

For our theorems, it will be essential to consider additive systems containing an infinite number of sets.

Example~\ref{ex:decimal} and numeral systems in general can be turned into additive systems for $\Z$, as the following well-known examples show.

\begin{example}[negabinary numeral system]
\label{ex:-bin}
The collection $\{A_i\}_{i\in\N_0}$ with $A_i=(-2)^i[2]$ is an additive system for $\Z$.
\end{example}

\begin{example}[balanced ternary numeral system]
\label{ex:b3}
The collection $\{A_i\}_{i\in\N_0}$ with $A_i=3^i\{-1,0,1\}$ is an additive system for $\Z$. In other words, any integer $n$ can be written uniquely in the numeral system with base~3 using the digits 0, 1, and $-1$ with the sequence of digits $d_i$ (for any fixed $n$ and for $i$ large enough, $d_i=0$), where $3^i d_i = a_i\in A_i$. E.g.\ $-35 = ...+0\cdot81+(-1)\cdot27+(-1)\cdot9+0\cdot3+1\cdot1$.
\end{example}

Examples \ref{ex:-bin} and \ref{ex:b3} motivate the following definition.

\begin{definition}
\label{def:main}
Let $(b_i)_{i\in\N_0}$ be a sequence of integers $b_i>1$, called \emph{bases}. Let $(T_i)_{i\in\N_0}$ be a sequence of sets such that $T_i$ is a complete residue system modulo $b_i$ that includes $0$ for each $i$. A \emph{canonical collection} is a collection of sets which for some sequences $(b_i)$ and $(T_i)$ has the form
\[
\{T_0,\ b_0T_1,\ b_0b_1T_2,\ b_0b_1b_2T_3,\ b_0b_1b_2b_3T_4,\ ...\}.
\]
A canonical collection is called \emph{complete} if it is an additive system for $\Z$.
\end{definition}
In fact, it is enough to specify the sequence $(T_i)$ because $b_i=|T_i|$ for all $i$.

The concept of a canonical collection is also illustrated in Example~\ref{ex:vis:canonical} below.

Note the following facts about canonical collections.
\begin{enumerate}
\item \label{p:dir} The sumset of any canonical collection is direct.
	\begin{proof}
	Let $n = a_0+a_1+\dotsb +a_k$ where $a_i\in A_i = b_0b_1\dotsb b_{i-1}T_i$ for all $i$. Then all $a_i$s can be determined from $n$ uniquely: by induction, suppose that $a_0,a_1,\dots,a_{i-1}$ are already known for some $i$; then $ n - (a_0+\dotsb+a_{i-1})$ has to be a multiple of $b_0b_1\dotsb b_{i-1}$ (because so is $a_{i'}$ for every $i'\geqslant i$) and $a_i$ has to be chosen as the unique element of $A_i$ with the same residue modulo $b_0b_1\dotsb b_{i}$ (because $a_{i''}$ is a multiple of $b_0b_1\dotsb b_{i}$ for every $i''>i$).
	\end{proof}
\item The additive systems for $\Z$ from Examples \ref{ex:-bin} and \ref{ex:b3} are canonical collections, defined by $T_i=\{0,(-1)^i\}$ and $T_i=\{-1,0,1\}$ respectively.
\item \label{p:sumNZ} Let $\mathcal{A}=\{A_0,A_1,\dotsc\}$ be a canonical collection with $\bigoplus\mathcal{A}=\Z$. Let $k\in\N_0$
 and $N=b_0b_1\dots b_k$. Then $\{A_{k+1}, A_{k+2},\dotsc\}$ is an additive system for $N\Z$ and $\{A_0,\dotsc,A_k\}$ taken modulo $N$ is an additive system for $\Z/N\Z$.
\item When $T_i=[b_i]$ for all $i$, a canonical collection becomes a British number system.
\item While a British number system has a natural ordering of its elements, there is generally no such ordering for canonical collections when viewed as unordered sets of subsets of $\Z$, even though an ordering is used in defining it.
	\begin{nonexample}
	Consider a canonical collection with $b_0=2$, $T_0=3[2]$, $b_1=3$, $T_1=[3]$ and another canonical collection with $b'_0=3$, $T'_0=2[3]$, $b'_1=2$, $T'_1=[2]$ and all other data the same. Then $A_0=A'_1$ and $A_1=A'_0$, so these two collections are identical and there is no way to reconstruct the order in which their elements were constructed just from the elements themselves.
	\end{nonexample}
\item Not every canonical collection is complete, even if its sumset is unbounded both from above and below.
	\begin{nonexample}
	\label{ex:badnbin}
	Let $b_i=3$ and $T_i=\{-2,0,2\}$ for all $i\in\N_0$. Then the sumset of the corresponding canonical collection is $2\Z$, so this collection is not complete.
	\end{nonexample}
	Generally, even if the elements from the union of a canonical collection have no common factor, the sumset of that collection can be sparse and have gaps.
	In that case, using the sequences $(b_i)$ and $(T_i)$, one can still formally write down the expansion $n=a_0+a_1+\dotsb$ (as noted in the proof for Fact \ref{p:dir}) even for $n$ not in the sumset, but that expansion will not eventually settle to zero.
\item Not every additive system for $\Z$ is a contraction of a canonical collection (i.e.\ there is no analogue of de~Bruijn's theorem for the integers).
	\begin{nonexample}[the smallest so-called tame aperiodic canon \cite{aperiodic}]
	\label{ex:aperiodic}~

	Let $A=\{0, 1, 5, 6, 12, 25, 29, 36, 42, 48, 49, 53\}$. The collection on the right-hand side of
	\[
	\Z = A \oplus 18[2] \oplus 8[3] \oplus 72\Z
	\]
	is not a canonical collection, nor a contraction of one. Indeed, suppose otherwise; $A$ has no common factor, therefore it is either $A_0=T_0$ itself or its sumset with some other $A_i$s, and $b_0|12$; the residues modulo $b_0$ in any direct sumset that includes $A_0$ have to be uniformly distributed, which means $b_0=2$ and then necessarily $A_0=\{0,1\}$; however then the odd half of $A$ would be a translation of the even half of $A$ by 1, which is false.
	\end{nonexample}

\begin{figure}
\centering
\begin{tikzpicture}
\newdimen\ybot
\newdimen\ytop
\newdimen\th
\newdimen\ra
\newdimen\ybotf
\newdimen\ytopf
\newdimen\xaz
\newdimen\gap
\ybot=2cm
\ytop=6cm
\th=0.05cm
\ra=0.4cm
\ybotf=3.3cm
\ytopf=5.5cm
\xaz=13cm
\gap=0.03cm

\draw[rounded corners=\ra, line width=\th, Blue4] (0, 0) rectangle (15.5cm, \ytop); 
\draw[rounded corners=\ra-2*\th-2*\gap, line width=\th, Blue2] (2*\th+2*\gap, \ybot-\th-\gap) rectangle (15.5cm-\th-\gap, \ytop-2*\th-2*\gap); 
\draw[rounded corners=\ra-3*\th-3*\gap, line width=\th, black] (3*\th+3*\gap, \ybot) rectangle (8cm-\th-\gap, \ytop-3*\th-3*\gap); 
\draw[rounded corners=\ra-4*\th-4*\gap, line width=\th, gray] (4*\th+4*\gap, \ybot+\th+\gap) rectangle (4cm-\th, \ytop-4*\th-4*\gap); 
\draw[rounded corners=\ra-\th-\gap, line width=\th, Green3] (\th+\gap, \th+\gap) rectangle (8cm, \ytop-\th-\gap); 
\draw[rounded corners=\ra-\th, line width=1.0*\th, Green4] (\xaz, -0.5cm) rectangle (16.5cm, \ytop+0.5cm); 
\draw[rounded corners=\ra, line width=\th, Gold4] (\xaz-1cm, \ybotf) rectangle (\xaz+2cm, \ytopf); 
\draw[line width=2*\th, dashed, black] (\xaz, \ybot+0.5*\th) -- (\xaz, \ytop-2.5*\th-2*\gap);
\draw[line width=3*\th, dotted, Orange1] (\xaz, \ybotf+0.5*\th) -- (\xaz, \ytopf-0.5*\th);

\node[gray, text width=4cm, align=center] at (2.2cm, \ytop-1.5cm) {numeral systems \\ $b$};
\node[gray, circle, fill, minimum size=3*\th, inner sep=0pt] (ExNS) at (1.0cm, \ybot+0.5cm) {};
\node[gray, above right=0cm of {ExNS}] {Ex.~\ref{ex:decimal}};
\node[text width=4cm, align=center] at (6cm, \ytop-1.5cm) {British \\ number systems \\ (Def.~\ref{def:british}) \\ $b_0,b_1,b_2,\dotsc$};
\node[circle, fill, minimum size=3*\th, inner sep=0pt] (ExBNS) at (4.5cm, \ybot+0.5cm) {};
\node[above right=0cm of {ExBNS}] {Ex.~\ref{ex:british}};
\node[text width=4cm, align=center, Blue2] at (10cm, \ytop-1.5cm) {canonical collections \\ (Def.~\ref{def:main}) \\ $b_0,b_1,b_2,\dotsc$ \\ $T_0,T_1,T_2,\dotsc$};
\node[circle, fill, minimum size=3*\th, inner sep=0pt, Blue2] (Exnbin) at (\xaz+0.5cm, \ybot+0.9cm) {};
\node[right=0cm of {Exnbin}, Blue2] {Ex.~\ref{ex:-bin}};
\node[circle, fill, minimum size=3*\th, inner sep=0pt, Blue2] (Exbtern) at (\xaz+0.5cm, \ybot+0.3cm) {};
\node[right=0cm of {Exbtern}, Blue2] {Ex.~\ref{ex:b3}};
\node[text width=10cm, align=center, Blue4, fill=white, fill opacity=0.7, text opacity=1] at (7.75cm, \ybot-1.5cm) {$\uparrow$ contractions of the above (Def.~\ref{def:contraction}) $\uparrow$};
\node[Green3, right=0cm of {(0.5cm, \ytop+0.5cm)}] {additive systems for $\N_0$, de Bruijn's theorem};
\node[Green4, right=0cm of {(\xaz-0.5cm, \ytop+1cm)}] {additive systems for $\Z$};
\node[circle, fill, minimum size=3*\th, inner sep=0pt, Blue2] (Nonex) at (8.5cm, \ybot+0.3cm) {};
\node[right=0cm of {Nonex}, Blue2] {Non-ex.~\ref{ex:badnbin}};
\node[circle, fill, minimum size=3*\th, inner sep=0pt, Green3] (Exaper) at (16cm, \ybot+1cm) {};
\node[above=0cm of {Exaper}, Green3] {\rotatebox{90}{Non-ex.~\ref{ex:aperiodic}}};
\node[text width=2.5cm, align=center, Gold4, below=1mm of {(\xaz+0.5cm, \ytopf)}, fill=white, fill opacity=0.7, text opacity=1, rounded corners=0.5cm, inner sep=1pt] {Fractran-type \\ canon.~coll. \\ (Def.~\ref{def:Fractran-type}) \\ $F$};
\end{tikzpicture}
\caption{\label{fig:Venn}
The sets of collections considered in this paper. A numeral system is specified by a base $b$, a British number system is specified by a sequence of bases $b_0,b_1,b_2,\dotsc$,~etc. In Section~\ref{sec:undecidable}, we prove Theorem~\ref{th:Collatz} concerned with the dashed black boundary and Corollary~\ref{cor:undecidable} concerned with the dashed orange boundary.}
\end{figure}

\item
Still, some sufficient conditions for a canonical collection to be an complete can be noted.
\begin{enumerate}
\item If every $T_i$ consists of consecutive integers then the corresponding canonical collection is complete if and only if $\max(T_i)>0$ for infinitely many $i$ and $\min(T_i)<0$ for infinitely many $i$. Examples \ref{ex:-bin} and~\ref{ex:b3} satisfy these conditions.
\item \label{suff:f} If $\mathcal{A}=\{A_0,A_1,\dotsc\}$ is a canonical collection and for some $k$ the sumset $A_k\oplus A_{k+1}\oplus \dotsb$ is equal to $N\Z$, where $N=b_0b_1\dotsb b_{k-1}$, then $\mathcal{A}$ is complete (the converse of Fact~\ref{p:sumNZ}). Therefore one cannot tell whether a canonical collection is complete or not by looking only at some finite initial portion of the sequences $(b_i)$ and $(T_i)$.
\begin{proof}
Any $n\in\Z$ can be uniquely decomposed into an element of $A_0\oplus\dotsb\oplus A_{k-1}$ and a multiple of $N$ (by the same argument as in the proof of Fact~\ref{p:dir}), and the latter is an element of $A_k\oplus A_{k+1}\oplus \dotsb$ by assumption.
\end{proof}
\end{enumerate}
\end{enumerate}

\pagebreak

The following visual example mirrors Example~\ref{ex:vis:British}.
\begin{example}
\label{ex:vis:canonical}
In a canonical collection $\mathcal{A}=\{A_0,A_1,\dotsc\}$, we set $b_0 = 2$ with $T_0=\{{0,1}\}$, $b_1=4$ with $T_1=\{{0,1,-2,3}\}$, and $b_2=3$ with $T_2=\{{0,1,-1}\}$.

The sets $T_0,~~T_1,~~T_2$:
\begin{figure}[h!]
\centering
\includegraphics[width=\linewidth]{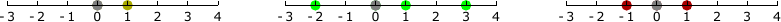}
\end{figure}

The sets $A_0=T_0,~~A_1=b_0T_1,~~A_2=b_0b_1T_2$:
\begin{figure}[h!]
\centering
\includegraphics[width=\linewidth]{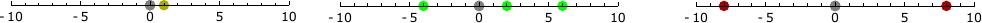}
\end{figure}

The sets $A_0,~~A_0+A_1,~~A_0+A_1+A_2$:
\begin{figure}[h!]
\centering
\includegraphics[width=\linewidth]{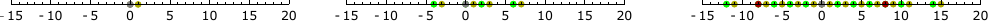}
\end{figure}

Note that these cumulative sumsets are also ever-expanding, as in the case of British number systems, even though there are gaps in them. These gaps are all eventually filled if $\mathcal{A}$ is truly an additive system for $\Z$. Moreover, any of these cumulative sumsets is a tile that tessellates $\Z$ via the one-dimensional lattice of translations generated by the translation by $b_0=2$, $b_0b_1=8$, and $b_0b_1b_2=24$, respectively.
\end{example}

\begin{remark}
In the more general case, when $\mathbb{M}$ is an arbitrary abelian group $G$, one can still define a canonical collection as follows. First we choose a (possibly left-infinite) subgroup series $\dotsb\subset G_2 \subset G_1 \subset G_0 \subset G_{-1}=G$. For each $i\in\N_0$, we form a subset of $G_{i-1}$ (and therefore $G$) by choosing a representative from each of the equivalence classes from $G_{i-1}/G_i$. These subsets of $G$ together form the canonical system.
\end{remark}

\begin{remark}
While the notation $\oplus$ is widespread, the terminology for the concepts introduced in Definition~\ref{def:basic} differs between sources. The term ``direct sumset'' is from \cite{book:ww}. The term ``additive system'' is from \cite{Nathanson}, where it was used specifically for $\mathbb{M}=\N_0$.
The term ``canonical collection'', introduced in Definition~\ref{def:main}, refers to \emph{periodic canons} that essentially underlie the canonical collections; that term, in turn, refers to canons in music \cite{aperiodic}.
\end{remark}

\section{Additive systems for $\Z$ are undecidable}
\label{sec:undecidable}
Section \ref{sec:canonical} suggests a natural question: given a canonical collection, how can we tell if it is complete, i.e.\ is an additive system for $\Z$?

There are uncountably many canonical collections, so one can expect questions like that to be complex. Still, it is remarkable that we can directly map this question to other well-known problems, as we will now show.

We will use the notation $r_b(n)$ for the (nonnegative) remainder on dividing $n$ by $b$.

Computing the decimal expansion of a number $n\in\N_0$ (recalling Example~\ref{ex:decimal}) is equivalent to repeated application of the map $f(n) = \frac{1}{10}(n - r_{10}(n))$, which removes the terminal digit of~$n$ (i.e.\ $f(n) = \lfloor n/10 \rfloor$), until $0$ is reached. More generally, computing the expansion of $n\in\N_0$ in a British number system is equivalent to computing the sequence
\begin{equation}
\label{eq:f-iter}
n,\ f_0(n),\ f_1\circ f_0(n),\ f_2\circ f_1\circ f_0(n),\ f_3\circ f_2\circ f_1\circ f_0(n), \dotsc
\end{equation}
until $0$ is reached, where $f_i(n) = \frac{1}{b_i}\left(n - r_{b_i}(n)\right)$ (i.e.\ $f_i(n) = \lfloor n/b_i \rfloor$).

More generally yet, suppose that the sequences $(b_i)$ and $(T_i)$ defining a canonical collection are given. For any integer $n$, let $t_i(n)$ be the unique element of $T_i$ that is congruent to $n$ modulo $b_i$.
Computing the (possibly infinite) expansion of $n\in\Z$ in that canonical collection is equivalent to computing the sequence from Eq.~\eqref{eq:f-iter} until $0$ is reached, this time with
\begin{equation}
\label{eq:fn}
f_i(n) = \frac{1}{b_i}\left(n - t_i(n)\right).
\end{equation}
Here $f_i$ removes the terminal digit of $n$ and adds an extra term that depends on the removed digit: 
\begin{equation}
\label{eq:fn2}
f_i(n) = \lfloor n/b_i \rfloor + f_i(r_{b_i}(n)).
\end{equation}
The floor function $\lfloor x\rfloor$ returns the greatest integer $\leqslant x$.

These maps together form a non-autonomous dynamical system, which is the point of view on the canonical collections that will be convenient for us from now on.
\begin{definition}
Let $(b_i)$ and $(T_i)$ be sequences that define a canonical collection, as in Definition~\ref{def:main}. The \emph{canonical system} associated with them is the sequence of maps $f_i$ defined by Eq.~\eqref{eq:fn}. For any $n\in\Z$, the sequence shown in Eq.~\eqref{eq:f-iter} is called the \emph{trajectory} of $n$ in the canonical system.
\end{definition}

Every canonical collection with given $(b_i)$ and $(T_i)$ corresponds to a canonical system via this definition, and conversely, every canonical system corresponds to a~canonical collection. The sets $T_i$ are related to the bases and the maps $f_i$ by $ t_i(j) = j-b_i f_i(j) $, i.e.\ instead of specifying $(T_i)$ for a canonical collection one can specify the values of $f_i(j)$ for $i\in\N_0$ and $j\in[n]$ arbitrarily, except for $f_i(0)=0$ for all $i$.

\begin{proposition}
Let $(b_i)$ and $(T_i)$ be sequences that define a canonical collection $\mathcal{A}=\{A_0,A_1,\ldots\}$ and the corresponding canonical system. Let $n\in\Z$. Then $n\in\bigoplus\mathcal{A}$ if and only if the trajectory of $n$ in the canonical system eventually becomes constant zero.
\end{proposition}
\begin{proof}
For $i\in\N_0$, we define $a_i = b_0b_1\dots b_{i-1}t_i(n_i)$, where $n_i = f_{i-1}\circ\dots\circ f_1\circ f_0(n)$ is the $i$th number in the trajectory. Then $a_i\in A_i$, moreover, $a_i$ are produced precisely by the procedure described in Fact~\ref{p:dir}. If $n\in\bigoplus\mathcal{A}$ then there exists some $I\in\N_0$ such that $n = a_0+a_1+\dots+a_{I-1}$, and in that case $n_I=0$ and $a_i=n_i=t_i(n_i)=0$ for all $i\geqslant I$, so the trajectory becomes constant zero. On the other hand, if the trajectory becomes constant zero, then the sequence $a_0,a_1,\dots$ has only a finite number of nonzero terms, so the sum $a_0+a_1+\dots$ is defined and is the desired decomposition of $n$ into elements of $A_i$.
\end{proof}
\begin{corollary}
\label{cor:coll-dyn}
A canonical collection is complete if and only if the trajectory of every integer in the corresponding canonical system eventually becomes constant zero.
\end{corollary}

The famously hard Collatz conjecture \cite{book:collatz} posits that any $n\in\N$ eventually enters the cycle $1\to4\to2\to1\to\dotsb$ under repeated application of $f_C\colon\,\N\to\N$ where $f_C(n)=3n+1$ if $n$ is odd and $f_C(n)=n/2$ if $n$ is even. The sequence $\left(f_C^i(n)\right)_{i\in\N_0}$ is called the \emph{Collatz trajectory} of~$n$.

\begin{theorem}
\label{th:Collatz}
There exists a canonical collection $\mathcal{A}_C$ such that $\mathcal{A}_C$ is complete if and only if the Collatz conjecture is true.
\end{theorem}
\begin{proof}
Let $\mathcal{A}_C$ be the canonical collection with $b_i=4^{i+1}$ for $i\in\N_0$ and $t_i(j) = j-b_if_C(j)$ for $j\in[b_i]\setminus[2]$ and $t_i(1) = 1$. The corresponding canonical system $(f_i)_{i\in\N_0}$ satisfies
\[
f_i(j) = 
\begin{cases}
0 & \text{if } j = 1, \\
j/2 & \text{if $j$ is even}, \\
3j+1 & \text{otherwise,}
\end{cases}
\]
for $j\in[b_i]$.
(Our specific choice of $b_i$ is not the only possible one for our purposes.)

We will now show that $\mathcal{A}_C$ has the property required in the theorem, using Corollary~\ref{cor:coll-dyn}.

Suppose that the $i$th value in a trajectory in this canonical system, $n_i$, falls into the interval $[b_i]$. Then the next value in the trajectory, $n_{i+1} = f_{i}(n_i)$, is either 0 (if $n_i\in[2]$) or $f_C(n_i)$, from the definition of $\mathcal{A}_C$. Therefore $n_{i+1} \leqslant f_C(n_i) \leqslant 3n_i+1 < 3b_i+1 < 4b_i = b_{i+1}$, so $n_{i+1}\in[b_{i+1}]$. By induction, at each further step $i'>i$, the trajectory value $n_{i'}$ will always be in the interval $[b_{i'}]$. Then the trajectory itself will proceed as the Collatz trajectory of $n_i$ unless and until it reaches 1, at which point it becomes constant zero. The Collatz conjecture states that it happens for any positive integer value of $n_i$.

It remains to prove that any number $n\in\Z$ ends up on a Collatz trajectory (by getting into the interval $[b_i]$ on the $i$th step for some $i$) and that every Collatz trajectory is reached from some starting number.

Suppose first that $n$ is positive. Assume it is the $i$th value in a trajectory (possibly $i=0$), i.e.\ $n=n_i$, and $n$ is not already in $[b_i]$, i.e.\ is of the form $n_i=b_i k_i + r_i$, $r_i= r_{b_i}(n_i)$ with $k_i\in\N$. Then, recalling Eq.~\eqref{eq:fn2},
\begin{multline*}
n_{i+1} = f_i(n) = k_i + f_C(r_i) \leqslant k_i + 3r_i+1 = n-r_i - (b_i-1)k_i + 3r_i+1 = \\ = n + 2r_i + 1 + k_i - b_i k_i < n  + 2b_i + k_i - b_i k_i = n + 2 - (k_i-2)(b_i-1).
\end{multline*}
If $k_i>2$ then $n_{i+1}<n$, so $k_{i+1}$ will be less than $k_i$ but will not become negative, so at some step $i'>i$ (repeating the argument for each step) we will get $k_{i'}\leqslant2$. If $k_i=2$ or $k_i=1$ then $n_{i+1}<b_{i+1}=4b_i$, as desired.

Suppose instead that $n$ is negative. We will show that it strictly increases with each iteration until it becomes nonnegative, which is then the case already considered in the previous paragraph. Again, let $n = n_i = b_ik_i+r_i$, $r_i = r_{b_i}(n_i)$ with $k_i$ a negative integer. Then $n_{i+1} = f_i(n) = k_i + f_C(r_i)$. Suppose $k_i<-1$, then $n_{i+1}\geqslant k_i \geqslant b_i(k_i+1) > n$, i.e.\ $n_{i'}$ will increase as a function of $i'\geqslant i$ until $k_{i'}$ becomes nonnegative (and then $n_{i'}$ is nonnegative too, as desired) or equal to $-1$. In the latter case, $n_{i'+1} = f_{i'}(n_{i'})>n_{i'}$ too, because either $n_{i'+1}\geqslant0$ or $r_{i'}\in[2]$, in which case $n_{i'+1}=-1>n_{i'}$.

Finally, the Collatz trajectory of any $n\in\N$ is reached by the canonical system corresponding to $\mathcal{A}_C$ starting from $n\cdot\prod_{i=0}^{\lceil\log_4n\rceil}b_i = n\cdot4^{1+2+\dotsb+(\lceil\log_4n\rceil+1)} = n\cdot2^{(\lceil\log_4n\rceil+1)(\lceil\log_4n\rceil+2)}$.
\end{proof}

This theorem shows that if a canonical collection is given, it can be very hard to tell if it is complete or not. Actually, even for a smaller, well-behaved subset of the set of all canonical collections, there is no procedure for identifying all those collections that are complete, as we will now show.

Fractran is a programming language which takes a single natural number as input; we will also allow the zero input, for which the program halts immediately. A Fractran program is a finite sequence of positive rational numbers. To execute a Fractran program, calculate the product of the input by each fraction from the sequence in turn and replace the input by the first resulting value that is an integer (this is a single \emph{Fractran iteration}), and then repeat this while possible (producing a \emph{Fractran trajectory}).

\begin{example}
The execution of the program
\[
\left( \frac{33}{20}, \frac{5}{11}, \frac{13}{10}, \frac{1}{5}, \frac{2}{3}, \frac{10}{7}, \frac{7}{2} \right)
\]
with input 8 proceeds as
\[
8 \xmapsto{\times\frac{7}{2}} 28 \xmapsto{\times\frac{10}{7}} 40 \xmapsto{\times\frac{33}{20}} 66 \xmapsto{\times\frac{5}{11}} 30 \xmapsto{\times\frac{13}{10}} 39 \xmapsto{\times\frac{2}{3}} 26 \xmapsto{\times\frac{7}{2}} 91 \xmapsto{\times\frac{10}{7}} 130 \xmapsto{\times\frac{13}{10}} 169
\]
and halts at that point. This program, given the $N$th power of 2 as input, outputs 13 raised to the power of the Hamming weight of $N$ \cite{stay}.
\end{example}

Conway introduced Fractran and showed that it is a computationally universal programming language \cite{book:Fractran}. In particular, this fact implies that the universal halting problem for Fractran (given a program, determine whether the program evaluated on every input eventually terminates) is undecidable; a problem that has finite data as input is \emph{undecidable} if there does not exist an algorithm that solves it for any valid input.

\begin{definition}
\label{def:Fractran-type}
Let $F=(F_1,\dotsc,F_m)$ be a Fractran program. Then let $f_F\colon\,\N_0\to\N_0$ be the function that applies a single Fractran iteration of that program to its input if it is possible or returns 0 otherwise (i.e.\ if the program immediately halts on that input).

We define a canonical collection $\mathcal{A}_F$ as follows. The sequence of bases is defined to be $b_i = b^{i+1}$, where $b = 1 + \left\lceil\max_{k=1}^mF_k\right\rceil$ (we want to ensure that this is an integer greater than~$1$). The sets $T_i$ are defined through the corresponding canonical system by setting $f_i(r)$ for $r\in[b_i-1]$ to be equal to $f_F(r)$; while for $r=b_i-1$ (which is to be referred as the \emph{exceptional case}) we set $f_i(r)=1$.

We will say that a canonical collection of the form $\mathcal{A}_F$ for some $F$ is \emph{Fractran-type}.
\end{definition}
Note that the family of Fractran-type canonical collections is well-behaved in the sense that specifying a Fractran-type canonical collection requires only a finite amount of data.

\begin{theorem}
\label{thm:Fractran-type}
A Fractran program $F$ halts for all inputs if and only if $\mathcal{A}_F$ is complete.
\end{theorem}
\begin{proof}
The proof is similar to that of Theorem~\ref{th:Collatz}.

Suppose that the $i$th value in a trajectory in the canonical system corresponding to $\mathcal{A}_F$, $n_i$, falls into the interval $[b_i-1]$. Then the definition of $\mathcal{A}_F$ ensures that the trajectory continues as the Fractran trajectory of $n_i$, only becoming constant zero if the actual Fractran trajectory terminates and the program $F$ halts on the input $n_i$. The number $n_i=b_i-1$ hits the exceptional case and continues to the Fractran trajectory of $1$.

It remains to prove that any number $n\in\Z$ ends up on a Fractran trajectory (by getting into the interval $[b_i]$ on the $i$th step for some $i$) and that every Fractran trajectory is reached from some starting number.

Suppose first that $n=n_i$ is positive. Suppose it is not already in that interval, i.e.\ is of the form $n=b_i k_i + r_i$, $r_i= r_{b_i}(n)$ with $k_i\in\N$. In the exceptional case $n$ is immediately sent to 1 and we are done, so suppose we are not in the exceptional case. Then
\begin{multline*}
n_{i+1} = f_i(n) = k_i + f_F(r_i) < k_i + br_i = n - (b_i-1)k_i + (b-1)r_i\leqslant \\ \leqslant n - (b_i-1)k_i + (b-1)(b_i-1) = n + (b_i-1)(b-1-k_i).
\end{multline*}
If $k_i\geqslant b-1$ then $n_{i+1}<n$, so $k_{i'}$ will decrease as a function of $i'\geqslant i$ (but will not become negative), so at some step we will get $k_{i'}<b-1$. While if $k_i<b-1$, we get $n_{i+1}<b_{i+1}=bb_i$, as desired.

Suppose instead that $n=n_i$ is negative. We will show that it strictly increases with each iteration until it becomes nonnegative, which is then the case already considered in the previous paragraph. Again, let $n = b_ik_i+r_i$, $r_i= r_{b_i}(n)$ with $k_i$ a negative integer. Then $n_{i+1} = f_i(n) = k_i + f_F(r_i)$ unless we are in the exceptional case. Suppose $k_i<-1$, then $n_{i+1}\geqslant k_i \geqslant b_i(k_i+1) > n$, i.e.\ $n_{i'}$ will increase as a function of $i'\geqslant i$ until $k_{i'}$ becomes nonnegative or equal to $-1$. In the latter case, if $n_{i'}=-1$, this is the exceptional case, $n_{i'+1} = f_{i'}(n_{i'}) = k_{i'} + 1 = 0$, and we are done, so suppose this is not the exceptional case. Then either $f_F(r_{i'})>0$ and $n_{i'+1}\geqslant0$ (and then we are done) or $f_F(r_{i'})=0$ and $n_{i'+1}=-1>n_{i'}$.

Finally, for $n\in\N$, let $L=\lceil\log_b (n+1)\rceil$, then the Fractran trajectory of $n$ is reached by the canonical system corresponding to $\mathcal{A}_F$ starting from $n\cdot\prod_{i=0}^{L-1}b_i = n\cdot b^{1+2+\dotsb+L} = n\cdot b^{L(L+1)/2}$. Indeed, the trajectory of that starting number reaches $n$ after $L$ steps and $b_L = b^L\cdot b \geqslant (n+1)b > n+1$, so the trajectory does not hit the exceptional case and continues to the Fractran trajectory.
\end{proof}

\begin{corollary}
\label{cor:undecidable}
Consider the problem of determining if the Fractran-type canonical collection $\mathcal{A}_F$ corresponding to a given Fractran program $F$ is complete. That problem is undecidable.
\end{corollary}
\begin{remark}
Results similar to the corollary above are known in the field of dynamical systems, such as \cite{mortality}.
In the context of tiling problems, undecidability results are known for problems like ``given one subset of a certain abelian group, find another one such that they together form an additive system for that group'', such as in \cite{grtao}.
\end{remark}

\section*{Acknowledgements}
This research is supported by EPSRC grant EP/W524098/1.
I am grateful to Matthew Lettington for introducing me to additive systems and to Ian Short, Matty van Son, and Katherine Staden for useful discussions.

\printbibliography

\end{document}